\documentclass[reqno]{amsart}

\usepackage{latexsym}
\usepackage{amsmath}
\usepackage{amssymb}
\usepackage{amscd}
\usepackage{multicol}
\usepackage{amsfonts}
\usepackage{amsthm}
\usepackage{color}
\usepackage{graphicx}
\usepackage{epsfig}
\usepackage{psfrag}

\input xy
\xyoption{all}

\theoremstyle{definition}
\newtheorem{thm}{Theorem}
\newtheorem{theorem}[thm]{Theorem}
\newtheorem{corollary}[thm]{Corollary}
\newtheorem{lemma}[thm]{Lemma}
\newtheorem{prop}[thm]{Proposition}
\newtheorem{defin}[thm]{Definition}
\newtheorem{remark}[thm]{Remark}

\newtheorem*{exm*}{Example}
\numberwithin{equation}{section}
\numberwithin{thm}{section}

\newcommand{\ft}{{\mathfrak{t}}}
\newcommand{\fb}{{\mathfrak{b}}}
\newcommand{\fh}{{\mathfrak{h}}}
\newcommand{\fg}{{\mathfrak{g}}}

\newcommand{\g}{\mathfrak{t}}
\newcommand{\gd}{\mathfrak{t}^{*}}

\renewcommand{\S}{{\mathbb{S}}}
\newcommand{\R}{\mathbb{R}}
\newcommand{\Z}{{\mathbb{Z}}}
\newcommand{\C}{{\mathbb{C}}}

\newcommand{\cM}{{\mathcal{M}}}

\newcommand{\G}{{\Gamma}}
\DeclareMathOperator{\maps}{Maps}
\DeclareMathOperator{\aut}{Aut}
\DeclareMathOperator{\hol}{Hol}

\title[$K$-theory of GKM bundles]{Equivariant $K$-theory of GKM bundles}
\author[V. Guillemin]{Victor Guillemin}
\address{Department of Mathematics, MIT, Cambridge, MA 02139}
\email{vwg@math.mit.edu}

\author[S. Sabatini]{Silvia Sabatini}
\address{Department of Mathematics, EPFL, Lausanne, Switzerland}
\email{silvia.sabatini@epfl.ch}

\author[C. Zara]{Catalin Zara}
\address{Department of Mathematics, University of Massachusetts
Boston, MA 02125}
\email{catalin.zara@umb.edu}

\date{January 31st, 2012}

\begin{document}

\begin{abstract}
Given a fiber bundle of GKM spaces, $\pi\colon M\to B$,
we analyze the structure of the equivariant $K$-ring of $M$ as a module
over the equivariant $K$-ring of $B$ by translating the fiber bundle, $\pi$,
into a fiber bundle of GKM graphs and constructing, by combinatorial techniques, a basis of
this module consisting of $K$-classes which are invariant under the natural holonomy action on the $K$-ring
of $M$ of the fundamental group of the GKM graph of $B$.
We also discuss the implications of this result for fiber bundles $\pi\colon M\to B$
where $M$ and $B$ are generalized partial flag varieties and show how our GKM description of
the equivariant $K$-ring of a homogeneous GKM space is related to the Kostant-Kumar description of this
ring.
\\$\;$
\\$\;$
 
\end{abstract}

\maketitle

\tableofcontents
\section{Introduction}
Let $T$ be an $n$-torus and $M$ a compact $T$ manifold.
The action of $T$ on $M$ is said to be GKM if $M^T$ is finite and if,
in addition, for every codimension one subtorus, $T'$, of $T$ the connected
components of $M^{T'}$ are of dimension at most $2$.
An implication of this assumption is that these fixed point components are either isolated
points or diffeomorphic copies of $S^2$ with its standard $S^1$ action,
and a convenient way of encoding this fixed point data is by means of the
\emph{GKM graph}, $\Gamma$, of $M$.
By definition, the $S^2$'s above are the \emph{edges} of this graph,
the points on the fixed point set, $M^T$, are the \emph{vertices} of this graph
and two vertices are joined by an edge if they are the fixed points for the
$S^1$ action on the $S^2$ representing that edge.
Moreover, to keep track of which $T'$'s correspond to which edges, one
defines a labeling function $\alpha$ on the set of oriented edges of $\Gamma$
with values in the weight lattice of $T$.
This function (the ``axial function" of $\Gamma$) assigns to each oriented edge the weight of the isotropy
action of $T$ on the tangent space to the north pole of the $S^2$ corresponding
to this edge. (The orientation of this $S^2$ enables one to distinguish the north pole
from the south pole.)

The concept of GKM space is due to Goresky-Kottwitz-MacPherson, who showed that the
equivariant cohomology ring, $H_T(M)$, can be computed from $(\Gamma,\alpha)$ (see \cite{GKM}).
Then Allen Knutson and Ioanid Rosu (see \cite{R}) proved the much harder result  
that this is also true for the
equivariant $K$-theory ring, $K_T(M)$. (We will give a graph theoretic description
of this ring in section \ref{K theory} below.)

Suppose now that $M$ and $B$ are GKM manifolds and $\pi\colon M\to B$ a $T$-equivariant
fiber bundle. Then the ring $H_T(M)$ becomes a module over the ring $H_T(B)$
and in \cite{GSZ1} we analyzed this module structure from a combinatorial perspective 
by showing that the fiber bundle $\pi$ of manifolds gives rise to a fiber bundle,
$\pi\colon \Gamma_M\to \Gamma_B$ of GKM graphs, and showing that the salient module structure
is encoded in this graph fiber bundle.
In this article we will prove analogous results in $K$-theory.
More explicitly in the first part of this article (from section \ref{K theory} to \ref{flag GKM}) we will review the basic facts
about GKM graphs, the notion of ``fiber bundle" for these graphs and the definition
of the $K$-theory ring of a graph-axial function pair, and we will also
discuss an important class of examples: the graphs associated with GKM spaces of the form $M=G/P$,
where $G$ is a complex reductive Lie group and $P$ a parabolic subgroup of $G$.
The main results of this paper
are discussed in section \ref{sec k-theory} and \ref{inv classes}.
In section \ref{sec k-theory} we prove that for a fiber bundle of GKM graphs
$(\Gamma,\alpha)\to(\Gamma_B,\alpha_B)$ a set of elements $c_1\ldots,c_k$
of $K_{\alpha}(\Gamma)$ is a free set of generators of $K_{\alpha}(\Gamma)$
as a module over $K_{\alpha_B}(\Gamma_B)$, providing their restrictions to
 $K_{\alpha_p}(\Gamma_p)$ are a basis for 
$K_{\alpha_p}(\Gamma_p)$, where $\Gamma_p$ is the graph theoretical fiber of $(\Gamma,\alpha)$
over a vertex $p$ of $\G_B$.
Then in section \ref{inv classes} we describe an important class of such generators.
One property of a GKM fiber bundle is a holonomy action of the fundamental group
of $\G_B$ on the fiber and we show how a collection of holonomy invariant generating classes
$c_1',\ldots,c_k'$ of $K_{\alpha_p}(\Gamma_p)$ extend canonically to a free set of 
generators $c_1,\ldots,c_k$ of $K_{\alpha}(\Gamma)$.
In section \ref{proj spaces} and \ref{An Cn} we describe how these
results apply to concrete examples: special cases of the $G/P$ examples mentioned above.
Finally in section \ref{more general} we relate our GKM description of the equivariant $K$-ring,
$K_T(G/P)$, to a concise and elegant alternative description of this ring by Bertram Kostant and Shrawan Kumar
in \cite{KK}, and analyze from their perspective the fiber bundle
$G/B\to G/P$.
\\$\;$

To conclude these introductory remarks we would like to thank Tudor Ratiu for his support
and encouragement and Tara Holm for helpful suggestions about the relations between GKM and 
Kostant-Kumar.
\section{$K$-theory of integral GKM graphs}\label{K theory}
Let $\Gamma=(V,E)$ be a $d$-valent graph, where $V$ is the set of vertices, and $E$
the set of oriented edges; for every edge $e\in E$ from $p$ to $q$, we denote by $\overline{e}$
the edge from $q$ to $p$. Let $i\colon E\to V$ (resp. $t\colon E \to V$) be the map which assigns
to each oriented edge $e$ its initial (resp. terminal) vertex (so $i(e)=t(\overline{e})$ and $t(e)=i(\overline{e})$);
for every $p\in V$ let $E_p$ be the set of edges whose initial vertex is $p$.

Let $T$ be an $n$-dimensional torus; we define a ``$T$-action on $\Gamma$"
by the following recipe.
\begin{defin}
Let $e=(p,q)$ be an oriented edge in $E$. Then a \textit{connection along $e$} is a bijection
 $\nabla_e\colon E_p\to E_q$ such that $\nabla_e(e)=\overline{e}$.
 A connection on the graph $\G$ is a family $\nabla=\{\nabla_e\}_{e\in E}$ satisfying
 $\nabla_{\overline{e}}=\nabla^{-1}_e$ for every $e\in \G$.
\end{defin}
Let  $\gd$ be the dual of the Lie algebra of $T$ and $\Z^*_T$ its weight lattice.
\begin{defin}\label{connection}
Let $\nabla$ be a connection on $\Gamma$. A \textit{$\nabla$-compatible integral axial function} on $\G$ is a map
$\alpha\colon E\to \Z^*_T$ satisfying the following conditions:
\begin{enumerate}
\item $\alpha(\overline{e})=-\alpha(e)$;
\item for every $p\in V$ the vectors $\{\alpha(e)\mid e\in E_p\}$ are pairwise linearly independent;
\item for every edge $e=(p,q)$ and every $e'\in E_p$ we have
$$
\alpha(\nabla_e(e'))-\alpha(e')=m(e,e')\alpha(e)\;,
$$
where $m(e,e')$ is an integer which depends on $e$ and $e'$.
\end{enumerate}
\end{defin}
An \textit{integral axial function} on $\G$ is a map $\alpha\colon E \to \Z^*_{T}$ which is $\nabla$-compatible,
for some connection $\nabla$ on $\G$.
\begin{defin}
An \textit{integral GKM graph} is a pair $(\Gamma,\alpha)$ consisting of a regular graph $\G$ and an integral axial
function $\alpha\colon E\to \Z_{T}^*$
\end{defin}
\begin{remark}
The graphs we described in the introduction are examples of such graphs.
In particular condition 2 in definition \ref{connection} is a consequence 
of the fact that, for every codimension one subgroup of $T$, its
fixed point components are of dimension at most two, and condition 3 a consequence
of the fact that this subgroup acts trivially on the tangent bundles of these
component.
\end{remark}
Observe that an integral GKM graph is a particular case of an abstract GKM graph, as defined in \cite{GSZ1};
here we require $\alpha$ to take values in $\Z_{T}^*$ rather than in $\gd$, 
and in definition \ref{connection} (3)
we require $\alpha(\nabla_e(e'))-\alpha(e')$ to be an integer multiple of $\alpha(e)$, for
every $e=(p,q)\in E$ and $e'\in E_p$. 
The necessity of these integrality properties will be clear from the definition of \textit{$T$-action on} $\G$.
Let $R(T)$ be the representation ring of $T$; notice that $R(T)$ can be identified with the
character ring of $T$, i.e. with the ring of finite sums 
\begin{equation}\label{rep ring}
\sum_k m_k e^{2\pi\sqrt{-1}\alpha_k}\;,
\end{equation}
where the $m_k$'s are
integers and $\alpha_k\in \Z^*_T$. So giving an axial function $\alpha\colon E \to \Z^*_{T}$
is equivalent to giving a map which assigns to each edge $e\in E$ a one dimensional representation $\rho_e$,
whose character $\chi_e\colon T \to S^{1}$ is given by 
$$\chi_e(e^{2\pi\sqrt{-1}\xi})=e^{2\pi\sqrt{-1} \alpha(e)(\xi)}\;.$$
For every $e\in E$, let $T_e=\ker(\chi_e)$, and consider the restriction map 
$$r_e\colon R(T)\to R(T_e)\;.$$
Then for every vertex $p\in V$, we also obtain a $d$-dimensional representation $$\nu_p\simeq \bigoplus_{e\in E_p}\rho_e$$
which, by definition \ref{connection} (3) satisfies
\begin{equation}\label{restriction repr}
r_e(\nu_{i(e)})\simeq r_e(\nu_{t(e)})\;.
\end{equation}
So an integral axial function $\alpha\colon E\to \Z_T^*$ defines a one-dimensional representation $\rho_e$
for every edge $e\in E$ and for every $p\in V$ a $d$-dimensional representation $\nu_p$ satisfying \eqref{restriction repr};
this is what we refer to as a \emph{$T$-action on} $\Gamma$.\\$\;$
\begin{remark}
Henceforth in this article all GKM graphs will, unless otherwise specified, be \emph{integral} GKM graphs.
\end{remark}
We will now define the $K$-ring $K_{\alpha}(\Gamma)$ of $(\Gamma,\alpha)$.
As we remarked in the introduction, Knutson and Rosu have proved that if 
 $(\G,\alpha)$ is the GKM graph associated to a GKM manifold $M$, then
$$
K_{\alpha}(\G)\simeq K_T(M)\;,
$$
where $K_T(M)$ is the equivariant $K$-theory ring of $M$ (cf. \cite{R}).

Let $\maps(V,R(T))$ be the ring of maps which assign to each vertex $p\in V$ a representation of $T$.
Following the argument in \cite{R},
we define a subring of $\maps(V,R(T))$, called the ring of \textit{$K$-classes of} $(\Gamma,\alpha)$.
\begin{defin}
Let $f$ be an element of $\maps(V,R(T))$. Then $f$ is  a $K$-class of $(\Gamma,\alpha)$ if 
for every edge $e=(p,q)\in E$
\begin{equation}\label{defin K}
r_e(f(p))=r_e(f(q))\;.
\end{equation}
\end{defin}
Observe that using the identification of $R(T)$ with the ring of finite sums \eqref{rep ring}, condition
\eqref{defin K} is equivalent to saying that for every $e=(p,q)\in E$
\begin{equation}\label{comp cond}
f(p)-f(q)=\beta(1-e^{2\pi\sqrt{-1} \alpha(e)}),
\end{equation}
for some $\beta$ in $R(T)$. 

If $f$ and $g$ are two $K$-classes, then also $f +g$ and $fg$ are; so the set of $K$-classes is a subring of
 $\text{Maps}(V,R(T))$.
\begin{defin}
The \emph{$K$-ring} of $(\Gamma,\alpha)$, denoted by $K_{\alpha}(\Gamma)$, is the subring of $\text{Maps}(V,R(T))$ consisting of all the $K$-classes.
\end{defin}
 
 
\section{GKM fiber bundles}
Let $(\G_1,\alpha_1)$ and $(\G_2,\alpha_2)$ be GKM graphs, where $\G_1=(V_1,E_1)$, $\G_2=(V_2,E_2)$, $\alpha_1\colon E_1 \to \Z_{T_1}^*\subset \gd_1$
and $\alpha_2\colon E_2\to \Z_{T_2}^*\subset \gd_2$. 
\begin{defin}
An \textit{isomorphism of GKM graphs}
$(\G_1,\alpha_1)$ and $(\G_2,\alpha_2)$ is
a pair $(\Phi,\Psi)$, where $\Phi\colon \G_1\to\G_2$
is an isomorphism of graphs, and $\Psi\colon \gd_1\to\gd_2$
is an isomorphism of linear spaces such that
 $\Psi(\Z_{T_1}^*)=\Z_{T_2}^*$, and for every edge
 $(p,q)$ of $\G_1$ we have $\alpha_2(\Phi(p),\Phi(q))=\Psi (\alpha_1(p,q))$.
\end{defin}
By definition, if $(\Phi,\Psi)$ is an isomorphism of GKM graphs, the diagram
\begin{equation}\label{iso commutes}
\xymatrix{
E_1 \ar[r]^{\Phi} \ar[d]_{\alpha_1} & E_2 \ar[d]_{\alpha_2} \\
\Z_{T_1}^*  \ar[r]^{\Psi|_{\Z_{T_1}^*}} &  \Z_{T_2}^*
 } \
\end{equation}
commutes. We can extend the map $\Psi$ to be a ring homomorphism from $R(T_1)$ to $R(T_2)$ using the
identification \eqref{rep ring} and defining $\Psi( e^{2\pi\sqrt{-1}\alpha})= e^{2\pi\sqrt{-1}\Psi(\alpha)}$,
for every $\alpha\in \Z_{T_1}^*$.

Given $f\in \maps(V_2, R(T_2))$, let $\Upsilon^*(f)\in\maps(V_1,R(T_1))$
be the map defined by $\Upsilon^*(f)(p)=\Psi^{-1}(f(\Phi(p)))$.
From the commutativity of the diagram \eqref{iso commutes}
it's easy to see that 
if $f\in K_{\alpha_2}(\G_2)$ then $\Upsilon^*(f)\in K_{\alpha_1}(\G_1)$, and
$\Upsilon^*$ defines an isomorphism between the two $K$-rings. 

We are now ready to define the main combinatorial objects of this paper, 
GKM fiber bundles. Let $(\G,\alpha)$ and $(\G_B,\alpha_B)$ be two GKM graphs, 
with $\alpha$ and $\alpha_B$ having images in the same weight lattice $\Z_T^*$. 
Let $\pi\colon \G=(V,E) \to \G_B=(V_B,E_B)$ be a surjective morphism of graphs. By that we mean 
that $\pi$ maps the vertices of $\G$ onto the vertices of $\G_B$, such that, for every edge 
$e=(p,q)$ of $\G$, either $\pi(p)=\pi(q)$ (in which case $e$ is called \emph{vertical}), 
or $(\pi(p), \pi(q))$ is an edge of $\G_B$ (in which case $e$ is called \emph{horizontal}). 
Such a morphism of graphs 
induces a map $(d\pi)_p \colon H_p \to E_{\pi(p)}$ from the set of horizontal edges at 
$p \in V$ to the set of all edges starting at $\pi(p) \in V_B$. The first condition we
impose for $\pi$ to be a GKM fiber bundle is the following:

\begin{description}
\item[1] For all vertices $p \in V$, $(d\pi)_p \colon H_p \to E_{\pi(p)}$ is a bijection compatible with the axial functions:
$$\alpha_B ((d\pi)_p(e))= \alpha(e), $$
for all $e=(p,q) \in H_p$.
\end{description}
The second condition has to do with the connection on $\G$ and $\G_B$.
\begin{description}
\item[2] The connection along edges of $\Gamma$ moves horizontal 
edges to horizontal edges, and vertical edges to vertical edges. Moreover, 
the restriction of the connection of $\Gamma$ to horizontal edges 
is compatible with the connection on $\G_B$.
\end{description}
For every vertex $p\in \G_B$, let $V_p=\pi^{-1}(p)\subset V$ and $\G_p$ 
the induced subgraph of $\G$ with vertex set $V_p$. If the map $\pi$ satisfies 
condition \textbf{1}, then, for every edge $e=(p,q)$ of $\G_B$, it induces a 
bijection $\Phi_{p,q} \colon V_p \to V_q$ by $\Phi_{p,q} (p') = q'$ if and only if
$(p,q) = (d\pi)_p(p',q')$.  
\begin{description}
\item[3] For every edge $(p,q)\in \G_B$, $\Phi_{p,q}\colon \G_p\to\G_q$
is an isomorphism of graphs compatible with the connection $\nabla$ on $\G$ in the following sense:
for every lift $e'=(p_1,q_1)$ of $e=(p,q)$ at $p_1$ and every edge $e''=(p_1,p_2)$ of $\G_p$
the connection along the horizontal edge $(p_1,q_1)$ moves the vertical edge $(p_1,p_2)$
to the vertical edge $(q_1,q_2)$, where $q_i=\Phi_{p,q}(p_i)$, $i=1,2$.
\end{description}

We can endow $\G_p$ 
with a GKM structure, which is just the restriction of the GKM structure 
of $(\Gamma,\alpha)$ to $\G_p$. The axial function on $\G_p$ 
is the restriction of $\alpha\colon E\to \Z_{T}^*$ to the edges of $\G_p$; 
we refer to it as $\alpha_p$, and it takes values in $\mathfrak{v}_p^*$, the subspace of 
$\gd$ generated by values of axial functions $\alpha(e)$, for edges $e$ of $\G_p$. 
The next condition we impose on $\pi$ is the following:

\begin{description}
\item[4] For every edge $(p,q)$ of $\G_B$, there exists an isomorphism of GKM graphs
$$
 \Upsilon_{p,q}=(\Phi_{p,q},\Psi_{p,q})\colon (\G_p,\alpha_p)\to (\G_q,\alpha_q)\;.
 $$
 %
  %
\end{description}
 By property (3) of definition \ref{connection} and the fact that $\alpha_B(e)=\alpha(e')$ we have 
 \begin{equation}
 \alpha(\Phi_{p,q}(p_1),\Phi_{p,q}(p_2))-\alpha(p_1,p_2)=m'(p_1,p_2)\alpha_B(e)\;,
 \end{equation}
 where $m'(p_1,p_2)$ is an integer; so by the commutativity of \eqref{iso commutes}
 we have
 \begin{equation}\label{Psi compatible}
 \Psi_{p,q}(\alpha(p_1,p_2))-\alpha(p_1,p_2)=m'(p_1,p_2)\alpha_B(e)\;.
 \end{equation}
 Condition \eqref{Psi compatible} defines a map $m'\colon E_p\to \Z$, where $E_p$ is
 the set of edges of $\G_p$.
The next condition is a strengthening of this.
\begin{description}
\item[5] There exists a function $m\colon \mathfrak{v}_p^*\cap \Z_T^*\to \Z$ such that
\begin{equation}\label{Psi iso}
 \Psi_{p,q}(x)=x+m(x)\alpha_B(p,q)\;.
 \end{equation}
\end{description} 
 Observe that if $\G_B$ is connected, then all the fibers of $\pi$ are isomorphic as GKM graphs.
 More precisely, for any two vertices $p,q\in V_B$ let $\gamma$ be a path in $\G_B$ from
 $p$ to $q$, i.e. $\gamma\colon p=p_0\to p_1\to \cdots\to p_m=q$.
 Then the map
 $$
 \Upsilon_{\gamma}=\Upsilon_{p_{m-1},p_{m}}\circ \cdots\circ\Upsilon_{p_0,p_1}\colon(\G_p,\alpha_p)\to(\G_q,\alpha_q)\;.
 $$
 defines an isomorphism between the GKM graphs $(\G_p,\alpha_p)$ and $(\G_q,\alpha_q)$.
 As observed before, this isomorphism restricts to an isomorphism between the two $K$-rings,
 $\Upsilon_{\gamma}^*\colon K_{\alpha_q}(\G_q)\to K_{\alpha_p}(\G_p)$, which is not an 
 isomorphism of $R(T)$-modules, unless the linear isomorphism 
 $$\Psi_{\gamma}=\Psi_{p_0,p_1}\circ\cdots\circ \Psi_{p_{m-1},p_m}\colon
 \mathfrak{v}_q^*\to \mathfrak{v}_p^*$$ 
 is the identity.
 
 Let $\Omega(p)$ be the set of all loops in $\G_B$ that start and end at $p$, i.e. the set
 of paths $\gamma\colon p_0\to p_1\to \cdots \to p_{m-1} \to p_m$ such that $p_0=p_m=p$.
 Then every such $\gamma$ 
 determines a GKM isomorphism $\Upsilon_{\gamma}$ of the fiber $(\G_p,\alpha_p)$.
 The \emph{holonomy group}
 of the fiber $(\G_p,\alpha_p)$ is the subgroup of the GKM isomorphisms of the fiber, $\aut(\G_p,\alpha_p)$, given by
  $$
 \hol_{\pi}(\G_p)=\{\Upsilon_{\gamma}\mid \gamma\in \Omega(p)\}\leqslant \aut(\G_p,\alpha_p)\;.
 $$ 
 
\section{Flag manifolds as GKM fiber bundles}\label{flag GKM}
In this section we will discuss some important examples of GKM fiber bundles
coming from generalized partial flag varieties.
 
Let $G$ be a complex semisimple Lie group, $\fg$ its Lie algebra, 
$\fh\subset\fg$ a Cartan subalgebra and $\ft\subset \fh$ a compact real form;
let $T$ be the compact torus whose Lie algebra is $\ft$.
Let $\Delta\subset \Z_T^*\subset \ft^*$ be the set of roots and
$$
\fg=\fh\oplus \bigoplus_{\alpha\in \Delta}\fg_{\alpha}
$$
the Cartan decomposition of $\fg$.
Let $\Delta_0=\{\alpha_1,\ldots,\alpha_n\}\subset \Delta$ be a choice of simple roots
and $\Delta^+$ the corresponding positive roots. 
The set of positive roots determines a Borel subalgebra $\fb\subset \fg$ given by
$$
\fb=\fh\oplus\bigoplus_{\alpha\in \Delta^+}\fg_{\alpha}\;.
$$
Let $\Sigma\subset \Delta_0$ be a subset of simple roots, and $B\leqslant P(\Sigma)\leqslant G$,
where $B$ is the Borel subgroup whose Lie algebra is $\fb$ and
$P(\Sigma)$
the parabolic subgroup of $G$ corresponding to $\Sigma$.
Then $M=G/P(\Sigma)$ is a (generalized) partial flag manifold. In particular, if $\Sigma=\emptyset$, then $P(\emptyset)=B$,
and  we refer to
$M=G/B$ as the (generalized) complete flag manifold.

The torus $T$ with Lie algebra $\ft$ acts on $M=G/P(\Sigma)$ by left multiplication on $G$;
this action determines a GKM structure on $M$ with GKM graph $(\Gamma,\alpha)$.
In fact, let $W$ be the Weyl group of $\fg$ and $W(\Sigma)$ the subgroup of $W$
generated by reflections across the simple roots in $\Sigma$, and $\langle \Sigma \rangle$
the positive roots which can be written as a linear combination of the roots in $\Sigma$.
Then the vertices of $\Gamma$, corresponding to the $T$-fixed points, are in bijection
with the right cosets 
$$
W/W(\Sigma)=\{vW(\Sigma)\mid v\in W\}=\{[v]\mid v\in W\}\;,
$$ 
where $[v]=vW(\Sigma)$ is the right $W(\Sigma)$-coset containing $v$.
Two vertices $[v]$ and $[w]$ are joined by an edge if and only if
there exists $\beta\in \Delta^+\setminus \langle \Sigma\rangle$ 
such that $[v]=[ws_{\beta}]$; moreover the axial function $\alpha$ on the edge
$e=([v],[vs_{\beta}])$ is given by $\alpha([v],[vs_{\beta}])=v\beta$; it's easy
to see that this label is well defined.
Moreover, given an edge $e'=([v],[vs_{\beta'}])$ starting at $[v]$, 
a natural choice of a connection along $e$ is given by
$\nabla_ee'=([vs_{\beta}],[vs_{\beta}s_{\beta'}])$.
So 
$$
\alpha(\nabla_ee')-\alpha(e')=vs_{\beta}\beta'-v\beta'=s_{v\beta}v \beta'-v \beta'=m(e,e') v \beta\;,
$$
where $m(e,e')$ is a Cartan integer; so this connection satisfies property (3)
of definition \ref{connection}.

Consider the natural 
$T$-equivariant 
projection $\pi\colon G/B\to G/P(\Sigma)$; this map induces a projection map
on the corresponding GKM graphs $\pi\colon (\Gamma,\alpha)\to (\Gamma_B,\alpha_B)$
given by $\pi(w)=[w]$ for every $w\in W$, where $[w]=wW(\Sigma)$.
As we proved in \cite[section 4.3]{GSZ1}, $\pi$ is a GKM fiber bundle.
Moreover, let $\Gamma_{0}$ be the GKM graph of the fiber containing the identity
element of $W$; then $$\hol_{\pi}(\Gamma_0)\simeq W(\Sigma).$$
The key ingredient in the proof of this is the identification of the weights of the $T$ action
on the tangent space to the identity coset $p_0$ of $G/P$ with the complement
of $\langle \Sigma \rangle$ in $\Delta^+$. Namely for any such root
$\alpha$ let $w\in W$ be the Weyl group element $w\colon \mathfrak{t}\to \mathfrak{t}$
associated with the reflection in the hyperplane $\alpha=0$.
Then for the edge $e=(p_0,wp_0)$ of the GKM graph of $G/P$ the GKM isomorphism
of $\G_{p_0}$ onto $\G_p$ is the isomorphism $\Phi_{p_0,p}$ associated with the left
action of $w$ on $G/B$ and the $T$ automorphism \eqref{Psi iso} is just the action of $w$
on $\mathfrak{t}$ given by compositions of reflections in the hyperplane
$\alpha_B(p,q)=0$ in $\mathfrak{t}$, for some horizontal edge $(p,q)$. 
(These results are also a byproduct of the identification of the GKM model of $K_T(G/P)$
with the Kostant-Kumar model which we will describe in section \ref{more general}).

\section{$K$-theory of GKM fiber bundles}\label{sec k-theory}
Given a GKM fiber bundle $\pi\colon (\G,\alpha)\to (\G_B,\alpha_B)$,
 we will describe the $K$-ring of $(\G,\alpha)$ in terms of the $K$-ring
of $(\G_B,\alpha_B)$.
 In the proof of the main theorem we will need the following technical lemma.
\begin{lemma}\label{indep weights}
Let $\alpha$ and $\beta$ be linearly independent weights in $\gd$, and  $P$ an element of $R(T)$ . If $1-e^{2\pi\sqrt{-1}\alpha}$ divides $(1-e^{2\pi\sqrt{-1}\beta})P$ then $1-e^{2\pi\sqrt{-1}\alpha}$ divides $P$.
\end{lemma}
\begin{proof}
Let $\alpha=m\alpha_1$, for some $m\in \Z$, where $\alpha_1$ is a primitive element of the weight lattice in $\gd$.
We can complete $\alpha_1$ to a basis $\{\alpha_1,\alpha_2,\ldots,\alpha_d\}$ of the lattice. Let
$x_j=e^{2\pi\sqrt{-1}\alpha_j}$ for all $j=1,\ldots,d$. \\Then by hypothesis $(1-x_1^m)$ divides $(1-x_1^nQ(x_2,\ldots,x_d))P(x_1,\ldots,x_d)$ for some non-constant polynomial $Q(x_2,\ldots,x_d)$ . Consider an element $\xi\in \g\otimes \C$ such that $x_1(\xi)^m=1$; then $(1-x_1^nQ(x_2,\ldots,x_d))P(x_1,\ldots,x_d)(\xi)=0$.
Since in general $(1-x_1^nQ(x_2,\ldots,x_d))(\xi)\neq 0$, this implies that $(1-x_1^m)$ divides $P(x_1,\ldots,x_d)$.
\end{proof}
For every $K$-class $f\colon V_B\to R(T)$,
define the pull-back $\pi^{*}(f)\colon V\to R(T)$ by $\pi^{*}(f)(q)=f(\pi(q))$. It's easy to check that 
 $\pi^{*}(f)$ is a $K$-class on $(\Gamma,\alpha)$.
So $K_{\alpha}(\Gamma)$ contains $K_{\alpha_B}(\G_B)$ as a subring, and the map $\pi^{*}\colon K_{\alpha_B}(\G_B)\to K_{\alpha}(\Gamma)$ gives $K_{\alpha}(\Gamma)$ the structure of a $K_{\alpha_B}(\G_B)$-module.
\begin{defin}
A $K$-class $h\in K_{\alpha}(\Gamma)$ is called \textit{basic} if $h\in \pi^*(K_{\alpha_B}(\G_B))$.
\end{defin}
We denote the subring of basic $K$-classes by $(K_{\alpha}(\Gamma))_{bas}$;
clearly we have $$(K_{\alpha}(\Gamma))_{bas}\simeq K_{\alpha_B}(\G_B)\;.$$

\begin{theorem}\label{leray-serre}
Let $\pi:(\Gamma,\alpha)\to (\G_B,\alpha_B)$ be a GKM fiber bundle, and let $c_1,\ldots,c_m$ be $K$-classes
on $\Gamma$ such that for every $p\in V_B$ the restriction of these classes to the fiber $\Gamma_p=\pi^{-1}(p)$
form a basis for the $K$-ring of the fiber.
Then, as $K_{\alpha_B}(\G_B)$-modules, $K_{\alpha}(\Gamma)$ is isomorphic to the free $K_{\alpha_B}(\G_B)$-module on $c_1,\ldots,c_m$.
\end{theorem}
\begin{proof}

First of all, observe that any linear combination of $c_1,\ldots,c_m$ with coefficients in $(K_{\alpha}(\Gamma))_{bas}\simeq K_{\alpha_B}(\G_B)$
is an element of $K_{\alpha}(\Gamma)$. Now we want to prove that the $c_i$'s are independent over $K_{\alpha_B}(\G_B)$.
In order to prove so, let
$
\sum_{k=1}^m \beta_k c_k=0
$ for some $\beta_1,\ldots,\beta_m\in (K_{\alpha}(\Gamma))_{bas}$. Let $\Gamma_p=\pi^{-1}(p)$ denote the fiber over $p\in B$, $\iota_p:\Gamma_p\to \Gamma$
the inclusion, and $\iota_p^{*}\colon K_{\alpha}(\Gamma)\to K_{\alpha}(\Gamma_p)$ the restriction to the $K$-theory of the fiber. Then  $
\sum_{k=1}^m\iota_p^{*}( \beta_k c_k)=0
$ for all $p\in B$. Since the $\beta_k$'s are basic $K$-classes, $\iota_p^{*}(\beta_k)$ is just an element of $R(T)$ for all $k$. But by assumption
$\{\iota_p^{*}(c_1),\ldots,\iota_p^{*}(c_m)\}$ is a basis of $K_{\alpha}(\Gamma_p)$; so  $\iota_p^{*}(\beta_k)=0$ for all $k=1,\ldots,m$, for all $p\in \G_B$, which implies that $\beta_k=0$ for all $k$. We need to prove that the free $K_{\alpha_B}(\G_B)$-module generated by $c_1,\ldots,c_m$ is $K_{\alpha}(\Gamma)$.

Let $c\in K_{\alpha}(\Gamma)$. Since the classes $\iota_p^{*}c_1,\ldots,\iota_p^{*}c_m$ are a basis for $K_{\alpha}(\Gamma_p)$, there exist
$\beta_1,\ldots,\beta_m\in \text{Maps}( B, R(T))$ such that
$$
c=\sum_{k=1}^m \beta_k c_k\; ;
$$
we need to prove that the $\beta_k$'s belong to $(K_{\alpha}(\Gamma))_{bas}$ for all $k$. In order to prove this, it is sufficient
to show that \eqref{comp cond} is satisfied for every edge $(p,q)$ of $\G_B$.
Let $e'=(p',q')$ be the lift of $(p,q)$ at $p'\in \Gamma_p$. Then 
\begin{align*}
c(q') - c(p\,') = & \sum_{k=1}^m (\beta_k(q) c_k(q') - \beta_k(p) c_k(p\,')) \\
 = & \sum_{k=1}^m (\beta_k(q)-\beta_k(p)) c_k(p\,') +
 \sum_{k=1}^m \beta_k(q) (c_k(q') - c_k(p\,')) \; .
\end{align*}
Since $c, c_1, \ldots ,c_m$ belong to $K_{\alpha}(\Gamma)$, by \eqref{comp cond} the differences $c(q') - c(p\,')$, $c_k(q') - c_k(p\,')$ are multiples of $(1-e^{2\pi\sqrt{-1}\alpha(e')})$, for all $k=1,\ldots, m$. Therefore, for all $p\,' \in \Gamma_{p}$,
$$ \sum_{k=1}^m (\beta_k(q)-\beta_k(p)) c_k(p\,') =(1-e^{2\pi\sqrt{-1} \alpha(e')}) \eta(p\,') \; ,$$
where $\eta(p\,') \in R(T)$. We will show that $\eta \colon \Gamma_{\!p} \to R(T)$ belongs to $K_{\alpha}(\Gamma_{\!p})$.

If $p\,'$ and $p\,''$ are vertices in $\Gamma_{\!p}$, joined by an edge $(p\,',p\,'')$, then
$$ \sum_{k=1}^m (\beta_k(q)-\beta_k(p)) (c_k(p\,'')-c_k(p\,')) =
(1-e^{2\pi\sqrt{-1}\alpha(e')}) (\eta(p\,'')- \eta(p\,')) \; .$$
Each $c_k$ is a $K$-class on $\Gamma$, so $c_k(p'')-c_k(p')$ is a multiple of $(1-e^{2\pi\sqrt{-1}\alpha(p',p'')})$, for all $k=1,\ldots, m$. Then $(1-e^{2\pi\sqrt{-1}\alpha(p',p'')})$ divides $(1-e^{2\pi\sqrt{-1}\alpha(e')}) (\eta(p\,'')- \eta(p\,'))$. But $\alpha(e')$ and $\alpha(p',p'')$ are linearly independent vectors. Therefore, by Lemma \ref{indep weights}, $(1-e^{2\pi\sqrt{-1}\alpha(p',p'')})$ divides $\eta(p'')- \eta(p')$, and so $\eta$ is a $K$-class on $\Gamma_{p}$.

Since the classes $\iota_p^{*}c_1,\ldots,\iota_p^{*}c_m$ are a basis for $K_{\alpha}(\Gamma_{\!p})$ there exist $Q_1,\ldots, Q_m \in R(T)$ such that
$$\eta= \sum_{k=1}^m Q_k \iota_p^*c_k\; .$$
Then
$$\sum_{k=1}^m (\beta_k(q) - \beta_k(p) - Q_k(1-e^{2\pi\sqrt{-1}\alpha(e')}) )\iota_p^* c_k = 0\;. $$
But $\iota_p^*c_1,\ldots,\iota_p^* c_m$ are linearly independent over $R(T)$, so
$$\beta_k(q) - \beta_k(p) = Q_k(1-e^{2\pi\sqrt{-1}\alpha(e')})  \;.$$
Since $\alpha(e')=\alpha(p',q')=\alpha_B(p,q)$, this
implies that $\beta_k \in K_{\alpha_B}(\G_B)$. Therefore every $K$-class on $\Gamma$ can be written as a linear combination of classes $c_1,\ldots,c_m$, with coefficients in $K_{\alpha_B}(\G_B)$.
\end{proof}

\section{Invariant classes}\label{inv classes}
Let $\pi\colon(\G,\alpha)\to(\G_B,\alpha_B)$ be a GKM fiber bundle, and let $(\G_p,\alpha_p)$ be
one of its fibers. We say that $f\in K_{\alpha_p}(\G_p)$ is an \emph{invariant class} if $\Upsilon_{\gamma}^*(f)=f$ for every
$\Upsilon_{\gamma}\in \hol_{\pi}(\G_p)$. We denote by $(K_{\alpha_p}(\G_p))^{\hol}$ the subring of $K_{\alpha_p}(\G_p)$
given by invariant classes.

Given any such class $f\in (K_{\alpha_p}(\G_p))^{\hol}$, we can extend it to be an element of $\maps\colon V\to R(T)$
by the following recipe: let $q$ be a vertex of $\G_B$, and $\gamma$ a path in $\G_B$ from $q$ to $p$. Let $\Upsilon_{\gamma}\colon
(\G_q,\alpha_q)\to(\G_p,\alpha_p)$ be the isomorphism of GKM graphs associated to $\gamma$;
then, as we observed before, $\Upsilon_{\gamma}^*(f)$ defines an element of $K_{\alpha_q}(\G_q)$.
Notice that the invariance of $f$ implies that $\Upsilon_{\gamma}^*(f)$ only depends on the end-points of $\gamma$;
so we denote $\Upsilon_{\gamma}^*(f)$ by $f_q$.
\begin{prop}
Let $c\in\maps(V,R(T))$ be the map defined by 
$c(q')=f_{\pi(q')}(q')$ for any $q'\in V$. Then $c\in K_{\alpha}(\G)$. 
\end{prop}
\begin{proof}
Since $c$ is a class on each fiber, it is sufficient to check the compatibility condition \eqref{comp cond}
on horizontal edges; let $e=(p,q)$ be an edge of $\G_B$, and let $e'=(p',q')$ be its lift at $p'\in V$.
If $c(p')=f_p(p')=\sum_k n_k e^{2\pi\sqrt{-1}\alpha_k}$, where $\alpha_k\in \Z_T^*\cap \mathfrak{v}^*_p$ and $n_k\in \Z$ for all $k$, then
$$
c(p')-c(q')=f_p(p')-f_q(q')=f_p(p')-\Psi_e(f_p(p'))=\sum_k n_k (e^{2\pi\sqrt{-1} \alpha_k}-e^{2\pi\sqrt{-1} \Psi_e(\alpha_k)})\;.
$$
By definition of GKM fiber bundle, $\Psi_e(\alpha_k)=\alpha_k+c(\alpha_k)\alpha_B(p,q)$ for every $k$, where $c(\alpha_k)$ is an
integer and $\alpha_B(p,q)=\alpha(p',q')$. So $e^{2\pi\sqrt{-1} \alpha_k}-e^{2\pi\sqrt{-1} \Psi_e(\alpha_k)}=e^{2\pi\sqrt{-1}\alpha_k}(1-e^{2\pi\sqrt{-1}c(\alpha_k)\alpha_B(p,q)})$, and it's easy to see
that $(1-e^{2\pi\sqrt{-1}c(\alpha_k)\alpha_B(p,q)})= \beta_k(1-e^{2\pi\sqrt{-1}\alpha_B(p,q)})$ for some 
$\beta_k\in R(T)$, for every $k$.
\end{proof}

\section{Classes on projective spaces}\label{proj spaces}
Let $T=(S^1)^n$ be the compact torus of dimension $n$,
with Lie algebra $\ft= \R^n$, and
let $\{ y_1, \ldots , y_n\}$ be the basis of $\ft^* \simeq \R^n$
dual to the canonical basis
of $\R^n$. Let $\{e_1, \ldots, e_n\}$ be the canonical basis of
$\C^n$. The torus $T$ acts
componentwise on $\C^n$ by
\begin{equation}\label{action std}
   (t_1,\ldots,t_n) \cdot (z_1,\ldots, z_n) = (t_1 z_1, \ldots, t_n z_n) \; .
\end{equation}
This action induces a GKM action of $T$ on 
$\C P^{n-1}$, and the GKM graph is $\Gamma = \mathcal{K}_n$, the complete graph on $n$ vertices labelled by  $[n]=\{1, \ldots, n\}$. 
The axial function $\alpha$ on the edge $(i,j)$ is $y_i-y_j$, for every $i\neq j$.
Let
$\S= \Z[y_1, \ldots, y_n]$, $(\S)$ the field of fractions of $\S$, 
$\cM = \text{Maps}([n], \S)$, and
$$H_{\alpha}(\Gamma) = \{ f \in \cM \; | \; f(j)-f(k) \in (y_j-y_k)\S, \text{ for all } j \neq k\}\; .$$

Then $H_{\alpha}(\Gamma)$ is an $\S$-subalgebra of $\cM$. 
Let $\int_{\Gamma} \colon \cM \to (\S)$ be the map
$$\int_{\Gamma} f = \sum_{k=1}^n \frac{f(k)}{\prod_{j \neq k} (y_k-y_j)}\; .$$

\begin{prop}
Let $f \in \cM$. Then $f \in H_{\alpha}(\Gamma)$ if and only if $\int_{\Gamma}  f  \in \S$.
\end{prop}
\begin{proof}We have
$$\int_{\Gamma} f = \sum_{k=1}^n \frac{f(k)}{\prod_{j \neq k} (y_k-y_j)} = \frac{P}{\prod_{j<k} (y_j-y_k)}\; ,$$
where $P \in \S$. The factors in the denominator are distinct and relatively prime, hence $\int_\Gamma f \in \S$ if and only if all factors in the denominator divide $P$.

The factor $y_j-y_k$ comes from
\begin{align*}
\frac{f(k)}{\prod_{i \neq k} (y_i-y_j)} +& \frac{f(j)}{\prod_{i\neq j} (y_i-y_j)} = \frac{f(j)-f(k)}{(y_j-y_k) \prod_{i\neq j,k} (y_j-y_i)} + \\ & + \frac{f(k) \left( \prod_{i\neq k,k} (y_k-y_i) - \prod_{i \neq j,k} (y_j-y_i) \right)}{(y_j-y_k) \prod_{i \neq j,k} (y_j-y_i)(y_k-y_i)}\; .
\end{align*}

But $y_j-y_k$ divides the numerator of the second fraction, hence $y_j-y_k$ divides $P$ if and only if it divides $f(j)-f(k)$.
\end{proof}

The permutation group $S_n$ acts on $\S$ by permuting variables, 
and that action induces an action on $H_{\alpha}(\Gamma)$ by
$$(w \cdot f)(j) = w^{-1}\cdot f(w(j))\; .$$
We say that a class $f\in H_{\alpha}(\Gamma)$ is \emph{$S_n$-invariant} if
$w\cdot f = f$ for every $w\in S_n$, i.e.
$$
f(w(j))=w\cdot f(j)\quad\mbox{for every}\;\;w\in S_n\;.
$$
The goal of this section is to construct bases of the $\S$-module 
$H_\alpha(\Gamma)$ consisting of $S_n$-invariant classes, and to give explicit 
formulas for the coordinates of a given class in those bases.

Let $\phi \colon [n] \to \S$, $\phi(j) = y_j$ for all $1 \leqslant j \leqslant n$. Then 
$\phi$ is an $S_n$-invariant class in $H_{\alpha}(\Gamma)$. 
For $1 \leqslant k \leqslant n$, let $f_k = \phi^{k-1}$. Then $f_1, f_2, \ldots, f_{n}$ 
are $\S$-linearly independent invariant classes.

For $0 \leqslant j \leqslant n$, let $s_j$ be the $j^{th}$ elementary symmetric polynomial 
in the variables $y_1,\ldots, y_n$. Then $s_0 = 1$, $s_1 = y_1 + \dotsb + y_n$, 
$s_2 = y_1y_2 + y_1y_3 + \dotsb + y_{n-1}y_n$, \ldots, $s_n =  y_1 y_2 \dotsb y_n$. 
For $1 \leqslant k \leqslant n$, let 
$$g_k = f_k   - s_1 f_{k-1} + s_2 f_{k-2} - \dotsb +  (-1)^{k-1} s_{k-1} f_1\; .$$
Then $g_1, g_2, \ldots, g_{n}$ are invariant classes and the transition matrix from 
the $f$'s to the $g$'s is triangular with ones on the diagonal, hence it is invertible 
over $\S$. Therefore the classes $g_1, \ldots, g_n$ are also $\S$-linearly independent.

Let $\langle .,.\rangle \colon H_{\alpha}(\Gamma) \times H_{\alpha}(\Gamma) \to \S$
be the pairing
$$\langle f, g\rangle = \int_{\Gamma} fg\; .$$

\begin{thm}
The sets of classes $\{f_1, \ldots, f_n\}$ and $\{g_n, \ldots, g_1\}$ are dual to each other: 
$$\langle f_j , g_{n-k+1} \rangle = \delta_{jk}\; .$$
for all $1 \leqslant j,k \leqslant n$.
\end{thm}

\begin{proof}
We have $\int_{\Gamma} \phi^k = 0$ for all $0 \leqslant k \leqslant n\!-\!2$ and $\int_{\Gamma} \phi^{n-1}=1$. Moreover
$$\phi^n -s_1 \phi^{n-1} + s_2 \phi^{n-2} - \ldots + (-1)^ns_n \phi^0 = 0\; .$$

Let $1 \leqslant j,k \leqslant n$. Then
$$f_j g_{n-k+1} = \phi^{j-1} \left( \phi^{n-k} - s_1 \phi^{n-k-1} + \dotsb + (-1)^{n-k} s_{n-k} \phi^0 \right)\; .$$

If $j < k$, then 
$$f_j g_{n-k+1} = \text{ a combination of powers of } \phi \text{ at most } n-2\; ,$$
and then $\langle f_j , g_{n-k+1} \rangle  = 0$.

If $j=k$, then 
$$f_k g_{n-k+1} = \phi^{n-1} + \text{ a combination of powers of } \phi \text{ at most } n-2\; ,$$
hence $\langle f_k , g_{n-k+1} \rangle  = \int_{\Gamma} \phi^{n-1} = 1$.

If $j > k$, then
\begin{align*}
f_j g_{n-k+1} & = \phi^{j-k-1} \left( \phi^{n} - s_1 \phi^{n-1} + \dotsb + (-1)^{n-k} s_{n-k} \phi^k \right) = \\
& = - \,\phi^{j-k-1}  \left( (-1)^{n-k-1} s_{n-k+1} \phi^{k-1} + \ldots + (-1)^n s_n \phi^0 \right) = \\
& = \text{ a combination of powers of } \phi \text{ at most } n-2 \; ,
\end{align*}
and then $\langle f_j , g_{n-k+1} \rangle  = 0$.
\end{proof}

The following is an immediate consequence of this result.

\begin{corollary}
The sets $\{f_1, f_2, \ldots, f_{n} \}$ and $\{g_1, g_2, \ldots, g_{n} \}$ are dual bases 
of the $\S$-module $H_{\alpha}(\Gamma)$, and both bases consist of invariant classes. 
Moreover, if $h \in H_{\alpha}(\Gamma)$ and
$$h = a_1 f_1 + a_2 f_2 + \dotsb + a_n f_n = b_1 g_1 + b_2 g_2 \dotsb + b_n g_n\; ,$$
then
$$a_k = \langle g_{n-k+1}, h \rangle \quad \text{ and } \quad b_k = \langle f_{n-k+1}, h\rangle\; .$$
\end{corollary}

\medskip

The entire discussion above extends naturally to $K$-theory. Let $z_j = e^{2\pi\sqrt{-1}y_j}$ for
$j=1,\ldots, n$. 
If $y_1,\ldots,y_n$ denote a basis of $\Z_{T}^*$, we have that
$$R(T) = \Z[z_1, \ldots, z_n,z_1^{-1},\ldots, z_n^{-1}].$$

Let $\psi \colon \S \to R(T)$ be the injective ring morphism determined by $\psi(y_j) = z_j$, for $j=1,\ldots, n$. Its image is  $\psi(\S) = R_+(T)=\Z[z_1,\ldots, z_n]$. Let 
$$K_\alpha(\Gamma) = \{ g \colon [n] \to R(T) \; | \; f(j)-f(k) \in (z_j-z_k)R(T), \text{ for all } j \neq k\}\; .$$

Then 
$$\Phi \colon H_{\alpha}(\Gamma) \to K_{\alpha}(\Gamma) \quad , \quad \Phi(f)(j) = \psi(f(j))\; $$
is an injective morphism of rings and 
$$\Phi(qf) = \psi(q) \Phi(f)$$
for all $f \in H_{\alpha}(\Gamma)$ and $q \in \S$. The image of $\Phi$ is
$$\text{im} \Phi = \{ g \in K_{\alpha}(\Gamma) \; | \; \text{im}(g) \subset R_+(T)\}\; .$$
Let $\nu = \Phi(\phi)$; then $\nu \colon [n]\to R(T)$, $\nu(j) = z_j$.

The symmetric group $S_n$ acts on $R(T)$ by simultaneously permuting the variables $z$ and $z^{-1}$, 
and that action induces an action on $K_{\alpha}(\Gamma)$. The $K$-class $\nu$ is invariant, and so are its powers.

\begin{prop}\label{inv classes basis}
The invariant classes $\{1, \nu, \ldots, \nu^{n-1}\}$ form a basis of $K_{\alpha}(\Gamma)$ over $R(T)$.
\end{prop}

\begin{proof}
A Vandermonde determinant argument shows that the classes are independent. 
If $g \in K_{\alpha}(\Gamma)$, then there exists an invertible element $u \in R(T)$ such that $ug \in \text{im} \Phi$.
Let $ug = \Phi(f)$, with $f \in H_{\alpha}(\Gamma)$. If
$$f = a_0 \phi^0 +  \dotsb + a_{n-1}\phi^{n-1}\; ,$$
then
$$g = u^{-1} \psi(a_0) \nu^0 + \dotsb + u^{-1}\psi(a_{n-1}) \nu^{n-1}\, $$
and therefore the classes also generate $K_{\alpha}(\Gamma)$.
\end{proof}


\section{Invariant bases on flag manifolds}\label{An Cn}
In this section we use the same technique as above
to produce a basis of invariant $K$-classes on the variety of complete flags in $\C^{n+1}$.

As in the previous section, let $T=(S^1)^{n+1}$ act componentwise on $\C^{n+1}$ by 
$$
(t_1,\ldots,t_{n+1})\cdot (z_1,\ldots,z_{n+1})=(t_1z_1,\ldots,t_{n+1}z_{n+1}),
$$
and let $x_1,\ldots,x_{n+1}\in \Z_T^*$ be the weights.
This action induces a GKM action both on the flag manifold $Fl(\C^{n+1})$
and $\C P^n$.

The GKM graph $(\Gamma,\alpha)$ associated to $Fl(\C^{n+1})$ is the permutahedron:
its vertices are in bijection with the elements of $S_{n+1}$, the group of permutations
on $n+1$ elements, and there exists
an edge $e$ between two vertices $\sigma$ and $\sigma'$ if and only if
$\sigma$ and $\sigma'$ differ by a transposition, i.e. $\sigma'=\sigma(i,j)$,
for some $1\leqslant i< j\leqslant n+1$; the axial function is given by
$\alpha(\sigma,\sigma')=x_{\sigma(i)}-x_{\sigma'(i)}$.

The group $S_{n+1}$ acts on the GKM graph by left multiplication on its vertices,
and on $\gd$ by
$$
\sigma\cdot x_i=x_{\sigma(i)}.
$$
Using the identification \eqref{rep ring}, this action determines an action on $R(T)$ by defining
\begin{equation}\label{action on R(T)}
\sigma\cdot(e^{2\pi\sqrt{-1}x_i})=e^{ 2\pi\sqrt{-1}x_{\sigma(i)}}
\end{equation}
and then extending it to the elements of $R(T)$ in the natural way.
Using the fiber bundle construction introduced in this paper
we will now produce a 
basis of the $K$-ring of $(\Gamma,\alpha)$ composed of $S_{n+1}$-\emph{invariant classes},
i.e. elements $f\in K_{\alpha}(\Gamma)$ satisfying
$$
f(u)=u\cdot f(id)\quad\mbox{for every}\quad u\in S_{n+1}.
$$

The main idea in the construction of our invariant basis of $K$-classes for
$Fl(\C^{n+1})$ is to use the natural projection $\pi$ of $Fl(\C^{n+1})$ to $\C P^n$
with fiber $Fl(\C^n)$ to construct these classes by an induction argument on $n$.

As we saw in section \ref{proj spaces}, the GKM graph $(\Gamma_B,\alpha_B)$ associated to $\C P^n$ is $(\mathcal{K}_{n+1},\alpha_B)$, where 
$\alpha_B(i,j)=x_i-x_j$ for every $1\leq i\neq j\leq n+1$.
The projection $\pi\colon Fl(\C^{n+1})\to \C P^n$ can be described 
in a simple way in terms of the vertices of the GKM graphs; in fact
$\pi\colon (\Gamma,\alpha)\to (\Gamma_B,\alpha_B)$ is simply given
by $\pi(\sigma)=\sigma(1)$, for every $\sigma\in S_{n+1}$.

Let $z_i$ be $e^{2\pi\sqrt{-1}x_i}$ for every $i=1,\ldots,n+1$ and $\nu\colon [n+1]\to R(T)$ be invariant the $K$-class given by
$\nu(i)=z_i$ for $i=1,\ldots,n+1$.
By Proposition \ref{inv classes basis}, the classes $1,\nu,\ldots,\nu^{n}$ form an invariant basis of $K_{\alpha_B}(\mathcal{K}_{n+1})$, and they lift to $S_{n+1}$-invariant basic classes on $\G$.

Let $\G_{n+1}$ be the fiber over $\C\cdot e_{n+1}$, where $\{e_1,\ldots,e_{n+1}\}$ denotes
the canonical basis of $\C^{n+1}$; the holonomy group on this
fiber is isomorphic to $S_n$ viewed as the subgroup of $S_{n+1}$ which leaves the element
$n+1$ fixed. Once we construct a basis of $K_{\alpha}(\G_{n+1})$ consisting of
holonomy invariant classes, we can extend those to a set of classes of $K_{\alpha}(\G)$,
which, together with the invariant basic classes constructed above, will generate a basis of
$K_{\alpha}(\G)$ as a module over $R(T)$. 
We will show that if we start with a natural choice for the base of the induction, then 
the global classes generated by this means are indeed $S_{n+1}$ invariant.


Letting $I=[i_1,\dots,i_n]$ be a multi-index of non-negative integers, define
$$
\mathbf{z}^I=z_1^{i_1}z_2^{i_2}\cdots z_n^{i_n}
$$
and let $C_I=C_T(\mathbf{z}^I)\colon S_{n+1}\to R(T)$ be the element defined by
$$
C_I(\sigma)=\sigma\cdot \mathbf{z}^I\quad\mbox{for every}\quad\sigma\in S_{n+1}\;;
$$
it's easy to verify that $C_I$ is an invariant class of $K_{\alpha}(\Gamma)$.
\begin{theorem}
Let $$\mathcal{A}_n=\{I=[i_1,\ldots,i_n]\mid 0\leqslant i_1\leqslant n,0\leqslant i_2\leqslant n-1,\ldots,0\leqslant i_n\leqslant 1\}.$$
Then the set $$
\{C_I=C_T(\mathbf{z}^I)\mid I\in \mathcal{A}_n\}
$$
is an invariant basis of $K_{\alpha}(\Gamma)$ as an $R(T)$-module.
\end{theorem}
\begin{proof}
As mentioned above, 
the proof is by induction. Let $n=2$, then the fiber bundle $\pi\colon Fl(\C^3)\to \C P^2$
is a $\C P^1$-bundle. 
Let $p=\C\cdot e_3\in \C P^2$ be the one dimensional subspace generated
by the third vector in the canonical basis of $\C^3$. Then the fiber over $p$
is a copy of $\C P^1$ and
the invariant classes $C_{[0]}$ and $C_{[1]}$ form a basis of the $K$-ring of the fiber.
We can extend these classes using transition maps between fibers, thus obtaining
$S_3$-invariant $K$-classes $C_{[0,0]}$ and $C_{[0,1]}$ on $Fl(\C^3)$.
By Proposition \ref{inv classes basis}, the invariant classes $1,\nu,\nu^2$ form 
a basis of the $K$-ring of the base, and they lift to $S_3$-invariant basic classes
$C_{[1,0]}$ and $C_{[2,0]}$.
By Theorem \ref{leray-serre}, the $K$-ring of $Fl(\C^3)$ is freely generated over $R(T)$ by
the invariant classes
$C_{[0,0]}$, $C_{[0,1]}$, $C_{[1,0]}$, $C_{[1,1]}$, $C_{[2,0]}$ and $C_{[2,1]}$.

The general statement follows by repeating inductively the same argument, since at each stage the fiber
of $\pi\colon Fl(\C^{m+1})\to \C P^m$ is a copy of $Fl(\C^m)$.
\end{proof}
\begin{remark}
This argument can be adapted to give a basis of invariant $K$-classes for the
generalized flag variety of type $C_n$.
Let $\alpha_i=x_i-x_{i+1}$ for $i=1,\ldots,n-1$ and $\alpha_n=2x_n$ be a choice
 of simple roots of type $C_n$. The corresponding Weyl group $W$
 is the group of signed permutations of $n$ elements.
Let $\Sigma=\{\alpha_2,\ldots,\alpha_n\}$, then $G/P(\Sigma)$ is a GKM manifold
diffeomorphic to a complex projective
space $\C P^{2n-1}$; its GKM graph is a complete graph
$\mathcal{K}_{2n}$ whose vertices can be identified with the set
$\{\pm 1,\ldots,\pm n\}$ and the axial function $\alpha$ is simply given
by $\alpha(\pm i,\pm j)=\pm x_i\mp x_j$, for every edge $(\pm i,\pm j)$
of the GKM graph.
Observe that the procedure in section \ref{proj spaces} can be used here
to produce a $W$-invariant basis of $K_{\alpha}(\mathcal{K}_{2n})$. In fact
it is sufficient to let $y_1=x_1,\ldots,y_n=x_n,y_{n+1}=-x_1$ and $y_{2n}=-x_n$;
the basis of Proposition \ref{inv classes basis} is $S_{2n}$-invariant, and hence
in particular $W$-invariant.

Let $G/B$ be the generic coadjoint orbit of type $C_n$, and $(\Gamma,\alpha)$
its GKM graph.
If we consider the natural projection $G/B\to G/P(\Sigma)$, the fiber is
diffeomorphic to a generic coadjoint orbit of type $C_{n-1}$. Hence we can repeat the
inductive argument used in type $A_n$ to produce a $W$-invariant
basis of $K_{\alpha}(\Gamma)$.

\end{remark}
\section{The Kostant-Kumar description}\label{more general}
The manifolds in section \ref{flag GKM} are also describable in terms of compact groups. Namely, if we let $G_0$ be the compact
form of $G$ and $K$ the maximal compact subgroup of $P$, then $M=G/P=G_0/K$. Moreover,
$W=W_{G_0}$ and $W(\Sigma)=W_K$, $W_{G_0}$ and $W_K$ being the Weyl groups of $G_0$ and $K$, so $M^T=W_{G_0}/W_K$.
A fundamental theorem in equivariant $K$-theory is the Kostant-Kumar theorem, which asserts that $K_T(M)$ is isomorphic
to the tensor product 
\begin{equation}\label{KK}
R^{W_K}\otimes_{R^W}R\;,
\end{equation}
where $R$ is the character ring $R(T)$, and $R^{W_K}$ and $R^W$ are the subrings of $W_K$ and $W$-invariant
elements in $R$. This description of $K_T(M)$ generalizes to $K$-theory the well-known Borel description of the
equivariant cohomology ring $H_T(M)$ as the tensor product 
\begin{equation}\label{eKK}
\S(\gd)^{W_K}\otimes_{\S(\gd)^W} \S(\gd)
\end{equation}
and in \cite{GHZ} the authors showed how to reconcile this description with the GKM description of $H_T(M)$.
Mutatis mutandi, their arguments work as well in $K$-theory and we will give below a brief description of the
$K$-theoretic version of their theorem.

Let $\Gamma$ be the GKM graph of $M$. As we pointed out above, $M^T=W/W_K$, so the vertices of
$\Gamma$ are the elements of $W/W_K$.
Now let $f\otimes g$ be a decomposable element of the tensor product \eqref{KK}.
Then one gets an $R$-valued function, $k(f\otimes g)$, on $W/W_K$ by setting
\begin{equation}\label{map k}
k(f\otimes g)(wW_K)=wfg\;.
\end{equation}
One can show that this defines a ring morphism, $k$, of the ring \eqref{KK} into the ring $\maps(M^T,R)$,
and in fact that this ring morphism is a bijection of the ring \eqref{KK} onto $K_{\alpha}(\Gamma)$.
(For the proof of the analogous assertions in cohomology see section 2.4 of \cite{GHZ}.)
Moreover the action of $W$ on $K_T(M)$ becomes, under this isomorphism, the action
\begin{equation}\label{map w}
w(f_1\otimes f_2)=f_1\otimes wf_2
\end{equation}
of $W$ on the ring \eqref{KK}, so the ring of $W$-invariant elements in $K_T(M)$ gets identified with the tensor product
$
R^{W_K}\otimes_{R^W}R^W
$,
which is just the ring $R^{W_K}$ itself.
Finally we note that if $M$ is the generalized flag variety, $G/B=G_0/T$, \eqref{KK} becomes the tensor product
\begin{equation}\label{RR}
R\otimes_{R^W}R\;
\end{equation}
and the ring of $W$-invariant elements in $K_T(M)$ becomes $R$.
Moreover, if $\pi$ is the fibration $G_0/T\to G_0/K$, the fiber $F$ over the identity coset of $G_0/K$
is $K/T$; so $K_T(F)$ is the tensor product $R\otimes_{R^{W_K}}R$, and the subring of $W_K$-invariant elements in
$K_T(F)$ is $R$ which, as we saw above, is also the ring of $W_G$-invariant elements in $K_T(G_0/T)$ which is also
(see section \ref{inv classes}) the ring of invariant elements associated with the fibration $G_0/T\to G_0/K$.
Thus most of the features of our GKM description of the fibration $G_0/T\to G_0/K$ have simple interpretations
in terms of this Kostant-Kumar model.


\begin{thebibliography}{APK} 

\bibitem[GHZ]{GHZ} Guillemin, Victor, Tara Holm, and Catalin Zara.
A GKM description of the equivariant cohomology ring of a homogeneous space.
\emph{J. Algebraic Combin.} \textbf{23}, no. 1, 21--41, 2006.

\bibitem[GKM]{GKM} Goresky, Mark, Robert Kottwitz, and Robert MacPherson.
Equivariant cohomology, Koszul duality, and the localization theorem.
\emph{Invent. Math.} \textbf{131}, no. 1, 25--83, 1998.

\bibitem[GSZ1]{GSZ1} Victor Guillemin, Silvia Sabatini, Catalin Zara. Cohomology of GKM fiber bundles. \textit{J. Algebraic Combin.} \textbf{35}, no.1, 19--59, 2012.

\bibitem[GSZ2]{GSZ2} Victor Guillemin, Silvia Sabatini, Catalin Zara. Balanced fiber bundles and GKM theory. Preprint. 	arXiv:1110.4086v1[math.AT].

\bibitem[GZ]{GZ} Victor Guillemin, Catalin Zara. $G$-actions on graphs. \textit{International Mathematics Research Notes} No. 10,  519-542, 2001.

\bibitem[KK]{KK} Bertram Kostant, Shrawan Kumar. $T$-Equivariant $K$-Theory of Generalized Flag Varieties.\emph{ Journal of Differential Geometry}, \textbf{32}, (1990), 549-603.

\bibitem[R]{R} Ioanid Rosu. Equivariant $K$-theory and equivariant cohomology. With an appendix by Allen Knutson and Ioanid Rosu. \textit{Math. Z.} \textbf{243}, no. 3, 423-448, 2003.

\end{thebibliography}
\end{document}